\documentclass[11pt]{amsart}
\usepackage{amssymb}
\usepackage{verbatim}
\usepackage{hyperref}
\usepackage{color}

\begin{document}

% define theorem environments
\newtheorem{maintheorem}{Main Theorem}    %[section]
\newtheorem{theorem}{Theorem}    [section]
\newtheorem{proposition}[theorem]{Proposition}
\newtheorem{conjecture}[theorem]{Conjecture}
\def\theconjecture{\unskip}
\newtheorem{corollary}[theorem]{Corollary}
\newtheorem{lemma}[theorem]{Lemma}
\newtheorem{sublemma}[theorem]{Sublemma}
\newtheorem{fact}[theorem]{Fact}
\newtheorem{observation}[theorem]{Observation}
\theoremstyle{definition}
\newtheorem{definition}{Definition}
\newtheorem{notation}[definition]{Notation}
\newtheorem{remark}[definition]{Remark}
\newtheorem{question}[definition]{Question}
\newtheorem{questions}[definition]{Questions}
\newtheorem{example}[definition]{Example}
\newtheorem{problem}[definition]{Problem}
\newtheorem{exercise}[definition]{Exercise}

%\numberwithin{theorem}{section}
\numberwithin{definition}{section}
\numberwithin{equation}{section}

\def\reals{{\mathbb R}}
\def\torus{{\mathbb T}}
\def\heis{{\mathbb H}}
\def\integers{{\mathbb Z}}
\def\rationals{{\mathbb Q}}
\def\naturals{{\mathbb N}}
\def\complex{{\mathbb C}\/}
\def\distance{\operatorname{distance}\,}
\def\support{\operatorname{support}\,}
\def\dist{\operatorname{dist}\,}
\def\Span{\operatorname{span}\,}
\def\degree{\operatorname{degree}\,}
\def\kernel{\operatorname{kernel}\,}
\def\dim{\operatorname{dim}\,}
\def\codim{\operatorname{codim}}
\def\trace{\operatorname{trace\,}}
\def\Span{\operatorname{span}\,}
\def\dimension{\operatorname{dimension}\,}
\def\codimension{\operatorname{codimension}\,}
\def\nullspace{\scriptk}
\def\kernel{\operatorname{Ker}}
\def\ZZ{ {\mathbb Z} }
\def\p{\partial}
\def\rp{{ ^{-1} }}
\def\Re{\operatorname{Re\,} }
\def\Im{\operatorname{Im\,} }
\def\ov{\overline}
\def\eps{\varepsilon}
\def\lt{L^2}
\def\diver{\operatorname{div}}
\def\curl{\operatorname{curl}}
\def\etta{\eta}
\newcommand{\norm}[1]{ \|  #1 \|}
\def\expect{\mathbb E}
\def\bull{$\bullet$\ }
\def\det{\operatorname{det}}
\def\Det{\operatorname{Det}}
\def\multiR{\mathbf R}
\def\bestA{\mathbf A}
\def\Apq{\mathbf A_{p,q}}
\def\rank{\mathbf r}
\def\diameter{\operatorname{diameter}}
\def\bp{\mathbf p}
\def\bff{\mathbf f}
\def\bg{\mathbf g}
\def\essd{\operatorname{essential\ diameter}}
\def\mab{\max(|A|,|B|)}

\newcommand{\abr}[1]{ \langle  #1 \rangle}

\newcommand{\Norm}[1]{ \Big\|  #1 \Big\| }
\newcommand{\set}[1]{ \left\{ #1 \right\} }
\def\one{{\mathbf 1}}
\newcommand{\modulo}[2]{[#1]_{#2}}

\def\scriptf{{\mathcal F}}
\def\scriptq{{\mathcal Q}}
\def\scriptg{{\mathcal G}}
\def\scriptm{{\mathcal M}}
\def\scriptb{{\mathcal B}}
\def\scriptc{{\mathcal C}}
\def\scriptt{{\mathcal T}}
\def\scripti{{\mathcal I}}
\def\scripte{{\mathcal E}}
\def\scriptv{{\mathcal V}}
\def\scriptw{{\mathcal W}}
\def\scriptu{{\mathcal U}}
\def\scriptS{{\mathcal S}}
\def\scripta{{\mathcal A}}
\def\scriptr{{\mathcal R}}
\def\scripto{{\mathcal O}}
\def\scripth{{\mathcal H}}
\def\scriptd{{\mathcal D}}
\def\scriptl{{\mathcal L}}
\def\scriptn{{\mathcal N}}
\def\scriptp{{\mathcal P}}
\def\scriptk{{\mathcal K}}
\def\scriptP{{\mathcal P}}
\def\scriptj{{\mathcal J}}
\def\frakv{{\mathfrak V}}
\def\frakG{{\mathfrak G}}
\def\frakA{{\mathfrak A}}

\begin{comment}
\def\frakg{{\mathfrak g}}
\def\frakG{{\mathfrak G}}
\def\boldn{\mathbf N}
\end{comment}

\author{Michael Christ}
\address{
        Michael Christ\\
        Department of Mathematics\\
        University of California \\
        Berkeley, CA 94720-3840, USA}
\email{mchrist@math.berkeley.edu}
\thanks{Research supported in part by NSF grant DMS-0901569 and by the
Mathematical Sciences Research Institute.}

%Any opinions, findings, and conclusions
%or recommendations expressed in this paper are those of the author
%and do not necessarily reflect the views of the National Science Foundation.}
%s DMS-0401260 and

\date{December 15, 2011}

\title {An Approximate Inverse Riesz-Sobolev Inequality}

\begin{abstract}
The Riesz-Sobolev inequality relates the convolution of nonnegative functions
with domains $\reals^d$ to the convolution of their symmetric nonincreasing rearrangements. 
We show that for dimension $d=1$, for indicator functions of sets, if the inequality is sufficiently close
to an equality then the sets in question must nearly coincide with intervals.
\end{abstract}

% MSC 26D15 looks closest to me
% 26D10 (52A40) taken from Burchard's paper
% 26A84 Brascamp-Lieb-Luttinger
% 26A86 Brascamp-Lieb
% 11B13 Additive bases, including sumsets [See also 05B10] 
% 11B25 Arithmetic progressions [See also 11N13] 
% 11P70 Inverse problems of additive number theory, including sumsets

\maketitle

\section{Introduction}

Let $|S|$ denote the Lebesgue measure of $S\subset\reals^d$.
Denote by $f^\star$ the equimeasurable symmetric nonincreasing rearrangement of a nonnegative measurable function $f$,
and if $|S|<\infty$, denote by $S^\star$ the ball $B$ centered at $0\in\reals^d$ which satisfies $|B|=|S|$.

Consider any nonnegative measurable functions $f,g,h$ defined on $\reals^d$
which tend to zero in the sense that for any $t>0$, $|\set{x: f(x)>t}|$ is finite, and the same holds for $g,h$.
The inequality of Sobolev and Riesz \cite{riesz},\cite{sobolev} states that
\begin{equation}\label{eq:RSforfns} \langle f*g,h\rangle \le \langle f^\star*g^\star,h^\star\rangle.\end{equation}
In particular, for indicator functions of measurable sets $A,B,C$ with finite Lebesgue measures,
\begin{equation}\label{eq:RS} \langle \one_A*\one_B,\one_C\rangle \le
\langle \one_{A^\star}*\one_{B^\star},\one_{C^\star}\rangle.  \end{equation}
This foundational inequality directly implies the formally more general \eqref{eq:RSforfns}.

Inverse theorems have been used to characterize those functions which extremize
certain specific inequalities. One element of Lieb's \cite{liebHLS} characterization of extremizers of the
Hardy-Littlewood-Sobolev inequality was the fact that
if $h=h^\star$, and if $h^\star$ is positive and strictly decreasing, 
then equality holds in \eqref{eq:RSforfns} only if $f=f^\star$ and $g=g^\star$ up to translations.
See for instance Theorem~3.9 in \cite{liebloss}.
Christ \cite{christradon} relied on a sharper inverse theorem of Burchard \cite{burchard}
to characterize extremizers for an inequality for the Radon transform. 
The simple one-dimensional case of Burchard's theorem states that if
\begin{equation} \label{HBu} \max(|A|,|B|,|C|)\  \le \  \min(|A|+|B|,\,|B|+|C|,\,|A|+|C|)  \end{equation}
then if equality holds in \eqref{eq:RS},
then $A,B,C$ must be intervals, up to null sets. 
Equality also implies that the centers $a,b,c$ of the intervals $A,B,C$ satisfy $a+b=c$.

In this paper we establish an inverse result which  describes cases of near equality  in \eqref{eq:RS} for $\reals^1$.
This will be applied in a companion paper \cite{christyoung} to characterize those functions which 
nearly extremize Young's convolution inequality for $\reals^d$.

Let $S\bigtriangleup T$ denote the symmetric difference between sets $S,T$.

\begin{theorem} \label{mainthm}
For any $\eps,\eps'>0$ there exists $\delta>0$ with the following property.
Let $A,B,E,F\subset\reals$ be Lebesgue measurable sets with positive, finite measures.
Suppose that 
\begin{equation} (1+\eps')\big(\max(|A|,|B|)-\min(|A|,|B|)\big) \le |E| \le \tfrac13 (1-\eps')(|A|+|B|) \end{equation}
and that $|F|=3|E|$.
If
\begin{align} \label{nearlysharp1}
&\langle \one_A*\one_B,\,\one_{E}\rangle \ge 
\langle \one_{A^\star}*\one_{B^\star},\,\one_{E^\star}\rangle -\delta\mab^2
\\
\intertext{and}
\label{nearlysharp2}
&\langle \one_A*\one_B,\,\one_{F}\rangle \ge 
\langle \one_{A^\star}*\one_{B^\star},\,\one_{F^\star}\rangle -\delta\mab^2
\end{align}
then there exists an interval $I\subset\reals$ such that
\begin{equation} |A\bigtriangleup I|<\eps |A|.  \end{equation}
\end{theorem}

The hypotheses of Theorem~\ref{mainthm} may benefit from clarifications.
Let $S_{t,A,B}$ denote the superlevel set 
\begin{equation} S_{t,A,B}=\set{x: \big(\one_A*\one_B\big)(x)>t}.\end{equation}
We often write $S_t$ as shorthand for $S_{t,A,B}$.
When $A,B$ are intervals,
$|S_{t,A,B}| \equiv |A|+|B|-2t$ for all $t\in [0,\norm{\one_A*\one_B}_\infty)$.
\begin{enumerate}
\item
The unexpected, and perhaps unsatisfactory, feature of this formulation is that
a lower bound for $\langle \one_A*\one_B,\one_S\rangle$ is hypothesized
for two sets $S$, rather than merely for a single set.
Worse yet, the measures of these two sets are required to be coupled. 
\item
The condition that $|F|=3|E|$ can be relaxed, for trivial reasons, to $|F|=3|E|+O(\delta\mab)$.
\item
The hypotheses are vacuous unless $\min(|A|,|B|)>\tfrac12\max(|A|,|B|)$. 
\item
In a companion paper \cite{christyoung} in which Theorem~\ref{mainthm} is applied, 
its hypotheses  are satisfied in a much more robust form. Indeed, 
\eqref{nearlysharp1} is known in that application to hold for a family of sets $E$
whose measures take on essentially all values in the range  $\max(|A|,|B|)-\min(|A|,|B|)<|E|<|A|+|B|$.
Thus the requirement that $|F|=3|E|$ is no encumbrance there.
The general form of the analysis in \cite{christyoung} suggests that this robust form of
the hypotheses might arise naturally in other applications, as well.
\item
Define $\alpha,\beta$ by $|E|=|A|+|B|-2\alpha$ and $|F|=|A|+|B|-2\beta$.
As will be proved below in Lemmas~\ref{lemma:Salphalowerbound} and \ref{lemma:Salphaupperbound}, 
it follows from the hypotheses \eqref{nearlysharp1}, \eqref{nearlysharp2} that 
\begin{align*}
&\big|\,|S_\alpha|-|E|\,\big|\le C\delta^{1/2}\mab,
\\ &\langle \one_A*\one_B,\,\one_{S_\alpha}\rangle \ge 
\langle \one_{A^\star}*\one_{B^\star},\,\one_{S_\alpha^\star}\rangle -C\delta^{1/2}\mab^2
\end{align*}
with corresponding statements for $S_\beta,F$.
\item
The hypothesis \eqref{nearlysharp2} involving $F$ can be replaced by its weaker consequence
\begin{equation}\label{Hbeta} |S_\beta| \le |A|+|B|-2\beta -\delta\mab \end{equation}
established in Lemma~\ref{lemma:Salphaupperbound}, where $\beta=\tfrac12(|A|+|B|-3|E|)$.
Taken at face value, \eqref{Hbeta} is an upper bound on $\one_A*\one_B$,
rather than a lower bound.
This seeming paradox hints at the structure of our analysis, which is related to the Brunn-Minkowski
inequality $|U+V|\ge |U|+|V|$. A well-known inverse principle is that if equality
holds in the Brunn-Minkowski inequality, then $U,V$ are equal to intervals, up to null sets. 
Here an approximate inverse principle, governing the case in which $|U+V|$ is relatively small, is exploited.
\item
Theorem~\ref{mainthm} continues to hold true if the relation $|F|=3|E|$ is generalized to $|F|=k|E|$
for an odd positive integer $k$, provided that $|F|\le (1-\eps')(|A|+|B|)$.
\end{enumerate}

Our analysis relies on an approximate inverse Brunn-Minkowski theorem, 
Proposition~\ref{prop:keystone} below, which at present seems
to be known only in dimension one. With the exception of this pivotal ingredient, the
analysis extends in a straightforward way to Euclidean space of arbitrary dimension, 
with the interval $I$ in the conclusion replaced by an ellipsoid.
We hope to establish a suitable approximate inverse Brunn-Minkowski theorem
for all dimensions in a subsequent paper, 
obtaining as a consequence an extension of Theorem~\ref{mainthm} to arbitrary dimensions.

A useful fact \cite{liebHLS} is that, under certain mild supplementary
assumptions, if $k=k^\star$ is unbounded, everywhere positive, and strictly decreasing,
and if $\langle f*g,k\rangle =\langle f^\star*g^\star,k\rangle$, then
$f=f^\star$ and $g=g^\star$ almost everywhere, up to translations.
The following extension, in which the hypothesis $k=k^\star$ is dropped at the expense
of a slightly stronger hypothesis on $k^\star$, follows directly from Theorem~\ref{mainthm}.

\begin{theorem} 
Let $h$ be a nonnegative function such that $|\set{x: h(x)>t}|<\infty$ for all $t>0$, and $|\set{x:h(x)>0}|>0$. 
Suppose that its symmetric nonincreasing rearrangement $h^\star$ is continuous and strictly decreasing
on its support.  Let $K$ be a compact subset of $(0,\norm{h}_\infty)$.

For any $\eps>0$ there exists $\delta>0$ with the following property.
Let $A,B\subset\reals$ be Lebesgue measurable sets with $|A|,|B|\in K$.
Suppose that  $\max(|A|,|B|)\le (2-\rho)\min(|A|,|B|)$.
If \begin{equation*} 
\langle \one_A*\one_B,\,h\rangle \ge 
\langle \one_{A^\star}*\one_{B^\star},\,h^\star\rangle -\delta\mab^2
\end{equation*}
then there exists an interval $I\subset\reals$ such that
\begin{equation} |A\bigtriangleup I|<\eps |A|.  \end{equation}
\end{theorem}
\noindent The constants in this result do depend on $h,K$.

The structure of the analysis is as follows: (i) The hypotheses of Theorem~\ref{mainthm}
imply a lower bound for $|S_\alpha|$ and an upper bound for $|S_\beta|$,
with $\alpha, \beta$ as in the statement. (ii) $S_\beta\supset S_\alpha-S_\alpha+S_\alpha$. Therefore
(i) becomes an upper bound for the measure of a sumset associated to $S_\alpha$.
(iii) An inverse theorem of additive combinatorics, concerning sets whose sumsets are small,
adapted to the continuum setting, implies that $S_\alpha$ nearly coincides with an interval. 
(iv) A compactness argument establishes the special case of Theorem~\ref{mainthm} in which the
set $E$ is nearly an interval. 

\section{On measures of superlevel sets of convolutions}

For $a\in(0,\infty)$ let $I_a=[-\tfrac12 a,\tfrac12 a]$.  Define 
\begin{equation} \Theta(a,b,c) = \langle \one_{I_a}*\one_{I_b},\one_{I_c}\rangle.  \end{equation}
This function $\Theta:(0,\infty)^3\to(0,\infty)$ is continuous, is strictly positive on $(0,\infty)^3$,
and is a symmetric function of its three arguments $a,b,c$.  If $0<b<a$ and $a-b<c<a+b$, then
\begin{equation}\label{Thetaexpression}
\Theta(a,b,c) = b(a-b)+\tfrac12\int_{a-b}^c (a+b-t)\,dt.  \end{equation}
In this section, we deduce certain bounds on the measures of superlevel sets
from the near equality $\langle \one_A*\one_B,\one_E\rangle\ge \Theta(|A|,|B|,|E|)-\delta\mab^2$.

Recall the notation $S_t=S_{t,A,B}=\set{x: (\one_A*\one_B)(x)>t}$.
\begin{lemma} \label{lemma:Salphalowerbound}
Let $A,B,E$ be Lebesgue measurable sets of finite, positive measures satisfying
\begin{gather}
\label{Hyp1}
\max(|A|,|B|)-\min(|A|,|B|)<|E|<|A|+|B| 
\\
\label{Hyp2}
\langle \one_A*\one_B,\one_E\rangle \ge \Theta(|A|,|B|,|E|)-\delta\mab^2\end{gather}
for some $\delta\in(0,1]$.  Define $\alpha$ by $|E|=|A|+|B|-2\alpha$.  Then 
\begin{equation}|S_{\alpha}\cap E| \ge |A|+|B|-2\alpha - C\delta^{1/2}\mab.\end{equation}
\end{lemma}
\noindent In particular, $|S_{\alpha}| \ge |A|+|B|-2\alpha - C\delta^{1/2}\mab$.

\begin{proof}
Set $f=\one_A*\one_B$.  Define $E'=E\cap S_\alpha$.
A simple calculation using \eqref{Thetaexpression} demonstrates that
\[ \Theta(|A|,|B|,|E'|)+\alpha (|E|-|E'|) + c\big(|E|-|E'|)^2 \le \Theta(|A|,|B|,|E|).  \] 
Consequently
\begin{align*}
\langle \one_A*\one_B,\one_E\rangle = \int_E f
&= \int_{E'}f +\int_{E\setminus E'}f
\\
&\le \Theta(|A|,|B|,|E'|) + \alpha|E\setminus E'|
\\
& \le \Theta(|A|,|B|,|E|)-c\big(|E|-|E'|)^2.
\end{align*}
Since by hypothesis $\int_E f\ge \Theta(|A|,|B|,|E|)-\delta\mab^2$, it follows that 
\[|E\setminus E'|^2=(|E|-|E'|)^2\le C\delta\mab^2,\]
so
\[ |E\setminus E'|\le C\delta^{1/2}\mab. \]
Therefore
\begin{multline} 
|S_\alpha\cap E|=|E'| \\ =|E|-|E\setminus E'| =|A|+|B|-2\alpha-|E\setminus E'|
\\ \ge |A|+|B|-2\alpha-C\delta^{1/2}\mab.  \end{multline}
\end{proof}

\begin{lemma} \label{lemma:Salphaupperbound}
Let $A,B,E$ be Lebesgue measurable sets of finite, positive measures satisfying \eqref{Hyp1} and \eqref{Hyp2}.
for some $\delta\in(0,1]$.  Define $\alpha$ by $|E|=|A|+|B|-2\alpha$.  Then 
\begin{equation} |E\bigtriangleup S_{\alpha}|\le C\delta^{1/2}\mab\end{equation}
and consequently
\begin{equation}|S_{\alpha}|\le |A|+|B|-2\alpha+2\delta^{1/2}\mab.\end{equation}
\end{lemma}

\begin{proof}
Consider any measurable set $S$ such that
$E\subset S\subset S_{\alpha}\cup E$ and $|S|\le |A|+|B|$.
Then 
\begin{align*} \langle \one_A*\one_B,\one_S\rangle & \ge \alpha|S\setminus E| + \langle \one_A*\one_B,\one_E\rangle
\\ &\ge \Theta(|A|,|B|,|E|)-\delta\mab^2 + \alpha (|S|-|E|).  \end{align*}
On the other hand, by the Riesz-Sobolev inequality and the integral formula \eqref{Thetaexpression}
for $\Theta$,
\[ \langle \one_A*\one_B,\one_S\rangle \le\Theta(|A|,|B|,|S|)
= \Theta(|A|,|B|,|E|)+\tfrac12\int_{|E|}^{|S|} (|A|+|B|-t)\,dt.  \]
Therefore
\[ \alpha(|S|-|E|)\le\delta\mab^2 +\tfrac12\int_{|E|}^{|S|} (|A|+|B|-t)\,dt, \]
which can be rewritten as
\[ \tfrac12 \int_{|E|}^{|S|} \big( t-(|A|+|B|-2\alpha)\big)\,dt \le \delta\mab^2.  \]
Since $|A|+|B|-2\alpha=|E|$, the left-hand side is simply $\tfrac14(|S|-|E|)^2$. 
Thus
\[ |S|\le |E|+2\delta^{1/2}\mab = |A|+|B|-2\alpha+2\delta^{1/2}\mab.  \] 
We conclude that $|S_{\alpha}\cup E|\le |A|+|B|-2\alpha+2\delta^{1/2}\mab$.
Since it was shown in the preceding lemma that
$|S_{\alpha}\cap E|\ge |A|+|B|-2\alpha-C\delta^{1/2}\mab$,
the required bound for $|S_\alpha\bigtriangleup E|$ follows.
\end{proof}

\begin{corollary}
Under the hypotheses of Lemmas~\ref{lemma:Salphalowerbound} and \ref{lemma:Salphaupperbound},
\begin{equation} \big|\, |S_\alpha|-|E|\, \big|\le C\delta^{1/2}\mab\end{equation}   and
\begin{equation} \langle \one_A*\one_B,\one_{S_\alpha}\rangle \ge \Theta(|A|,|B|,|S_\alpha|)-C\delta^{1/2}\mab^2.
\end{equation} \end{corollary}

\begin{proof} 
The first conclusion follows from our upper bound for $|S_\alpha\bigtriangleup E|$.
The final conclusion follows from the inequality
\[ \big|\langle \one_A*\one_B,\,\one_{S_\alpha}-\one_E\rangle\big|
\le \norm{\one_A*\one_B}_\infty|S_\alpha\bigtriangleup E| \le \min(|A|,|B|)C\delta^{1/2}\mab  \]
and the fact that the function $r\mapsto\Theta(|A|,|B|,r)$
is Lipschitz continuous with norm equal to $\mab$.  \end{proof}

\section{Additive structure of superlevel sets of convolutions}

For any sets $A,B$ define $A+B=\set{a+b: a\in A \text{ and } b\in B}$.
For any positive integers $\lambda,\mu$ and any set $S$ define
\begin{equation}\lambda S-\mu S=\set{\sum_{i=1}^\lambda x_i-\sum_{j=1}^\mu y_j: x_i,y_j\in S}; 
\end{equation}
define $\lambda S$ and $-\mu S$ by replacing the appropriate sums by zero.

The following result provides a criterion for a set to be contained in 
an interval of only slightly larger measure.
\begin{proposition} \label{prop:keystone}
Let $A\subset\reals^1$ be a Lebesgue measurable set with finite, positive measure.
If $|A+A|<3|A|$ then $A$ is contained in an interval of length $\le |A+A|-|A|$.
\end{proposition}

The proof is a straightforward reduction to a corresponding result for sums of finite sets
due to Freiman \cite{freiman}. It is deferred to \S\ref{section:keystone}.

Proposition~\ref{prop:keystone} is the only element of our analysis which does not extend
in a straightforward way to higher dimensions. Thus in order to establish the analogue of
Theorem~\ref{mainthm} in all dimensions, it would suffice to establish the analogue of this Proposition.

Let $U,V\subset\reals^1$ be Lebesgue measurable sets with finite measures.  Then
$|U\,\bigtriangleup\, V| + 2|U\cap V| = |U|+|V|$, and
$\norm{\one_U-\one_V}_1 = |U\,\bigtriangleup\, V|$. Therefore
\begin{equation} \norm{\one_U-\one_V}_1+2|U\cap V| = |U|+|V|.  \end{equation}
The triangle inequality for the $L^1$ norm has the following consequence.

\begin{lemma} \label{lemma:alphabeta1}
Let $A,B\subset\reals$ be measurable sets with finite, positive measures.
For $0<t<\min(|A|,|B|)$, consider the superlevel sets $S_t= \set{x\in \reals:\one_A*\one_B(x)> t}$ 
of the convolution product $\one_A*\one_B$.
Let $k$ be any positive integer, and let $\alpha_i>0$ for $1\le i\le 2k+1$.  Define $\beta$ by
\begin{equation}
\big(\beta-\tfrac{|A|+|B|}2\big) = \sum_{i=1}^{2k+1}\Big(\alpha_i -\tfrac{|A|+|B|}2\Big).
\end{equation}
Then 
\begin{equation} \label{eq:superlevelsetsums}
S_{\alpha_1}-S_{\alpha_2}+S_{\alpha_3}-S_{\alpha_4}+\cdots + S_{\alpha_{2k+1}} 
\subset S_\beta. \end{equation}
\end{lemma}
A corollary, by the one-dimensional Brunn-Minkowski inequality $|U+V|\ge |U|+|V|$, is that
\begin{equation} \big|S_\beta\big|\ge \sum_{i=1}^{2k+1}\big|S_{\alpha_i}\big|.  \end{equation}

\begin{proof}
To prove the inclusion, set $\tilde B=\set{z: -z\in B}$ and $A_x=\set{x+y: y\in A}$.
For any $t>0$, \[\set{x: (\one_A*\one_B)(x)>t} =\set{x:|A_x\cap \tilde B|>t}.\] 
Indeed,
\[\one_A*\one_B(x)=\int \one_A(x-y)\one_B(y)\,dy = \int \one_A(x+y)\one_{\tilde B}(y)\,dy = |A_x\cap \tilde B|.\]

For $x\in S_t$,
\begin{equation}
\norm{\one_{A_x}-\one_{\tilde B}}_1 = |A_x|+|\tilde B|-2|A_x\cap \tilde B| = |A|+|B|-2|A_x\cap \tilde B| <|A|+|B|-2t.
\end{equation}
Therefore by the triangle inequality, if $x\in S_{\alpha_1}$ and $x'\in S_{\alpha_2}$ then
\begin{equation} \norm{\one_{A_x}-\one_{A_{x'}}}_1  <2|A|+2|B|-2\alpha_1-2\alpha_2. \end{equation}
Since $\norm{\one_{A_x}-\one_{A_{x'}}}_1=\norm{\one_{A_{x-x'}}-\one_{A}}_1$,
\begin{equation} \norm{\one_{A_z}-\one_A}_1  <2|A|+2|B|-2\alpha_1-2\alpha_2
\text{ for any $z\in S_{\alpha_1}-S_{\alpha_2}$,}
\end{equation}
In the same way, for any $z\in S_{\alpha_1}-S_{\alpha_2}+S_{\alpha_3}-\cdots + S_{\alpha_{2k+1}}$,
\begin{equation*}
\norm{\one_{A_z}-\one_{\tilde B}}_1<(2k+1)(|A|+|B|)-2\sum_i \alpha_i
\end{equation*}
and consequently
\begin{multline} |A_z\cap \tilde B| =\tfrac12|A|+\tfrac12|B| -\tfrac12 \norm{\one_{A_z}-\one_{\tilde B}}_1 \\
> \tfrac12 |A| + \tfrac12 |B| -\tfrac{2k+1}2 |A|-\tfrac{2k+1}2 |B|+\sum_{i}\alpha_i = \beta.  \end{multline}
\end{proof}

A variant of Lemma~\ref{lemma:alphabeta1} follows from the same reasoning.
If $z\in S_\alpha-S_\alpha$ then $\norm{A_z-A}_1<2(|A|+|B|-2\alpha)$.
For any $z\in kS_\alpha-kS_\alpha$, $\norm{A_z-A}_1<2k(|A|+|B|-2\alpha)$.
Therefore
$|A_z\cap\tilde A| >\tfrac12|A|+\tfrac12|A|-\tfrac12(2k|A|+2k|B|-4k\alpha) = 2k\alpha-(k-1)|A|-k|B|$.
Thus $kS_\alpha-kS_\alpha\subset\set{x: (\one_A*\one_{\tilde A})(x)>\gamma}$ where $\gamma = 2k\alpha-(k-1)|A|-k|B|$.

\begin{corollary} \label{cor:nearlyintervals}
Let $A,B\subset\reals$ be Lebesgue measurable sets with finite, positive measures. 
For $t\ge 0$ define $S_t=\set{x: \one_A*\one_B(x)>t}$.  
Let $k$ be a positive integer,  and suppose that $\eps>0$ satisfies
\begin{equation}\label{eq:epsconstraint} (4k+1)\eps\mab \le |S_\alpha|.  \end{equation}
Let $\alpha\ge 0$.
Set $\beta= (2k+1)\alpha-k|A|-k|B|$, and assume that $\beta\ge 0$.  If  both
\begin{align} |S_\beta|&< |A|+|B|-2\beta+(2k+1)\eps\mab 
\\ |S_\alpha|&>|A|+|B|-2\alpha-\eps\mab \end{align}
then $S_\alpha$ is contained in some interval $I$ satisfying
\begin{equation} |I|<|S_\alpha|+(4k+2)\eps\mab. \end{equation}
Moreover,
\begin{align} |S_\alpha|&< |A|+|B| -2\alpha + (4k+1)\eps\mab  
\\
|S_\beta| &> |A|+|B|-2\beta-(2k+1)\eps\mab.
\end{align}
\end{corollary}

This conclusion is of interest primarily when $\eps\mab\ll |S_\alpha|$. 
It is trivial unless both $\alpha,\beta$ lie in the range $[0,\min(|A|,|B|))$.
For $k=1$, the only case which will be needed below, this range is nonvacuous if and only if
\begin{equation} \mab<2 \min(|A|,|B|), \end{equation}.

\begin{proof}
By the Brunn-Minkowski inequality,
\begin{align*}
|S_\beta|\ge |S_\alpha+S_\alpha| + |(k-1)S_\alpha - k S_\alpha|
\text{ and }
|(k-1)S_\alpha- kS_\alpha| \ge (2k-1)|S_\alpha|
\end{align*}
so
\begin{align*}
|S_\alpha+S_\alpha|&\le |S_\beta|-(2k-1)|S_\alpha|
\\
&\le 2|A|+2|B|-4\alpha + 4k\eps\mab
\\
&\le 2|S_\alpha|+(4k+1)\eps\mab.
\end{align*}
Since $(4k+1)\eps\mab \le |S_\alpha|)$, it follows from Proposition~\ref{prop:keystone} that 
$S_\alpha$ is contained in some interval $I$ whose length satisfies 
\begin{equation*} |I| < |S_\alpha|+(4k+1)\eps\mab.\end{equation*}

The inclusion $(k+1)S_\alpha-kS_\alpha\subset S_\beta$, together with 
the Brunn-Minkowski inequality, 
imply that $|S_\beta|\ge |(k+1)S_\alpha-kS_\alpha|\ge (2k+1)|S_\alpha|$.
The indicated lower bound for $|S_\beta|$ and upper bound for $|S_\alpha|$
follow from this relation together with the hypothesized upper and lower bounds for these same quantities.
\end{proof}

The Riesz-Sobolev inequality gives integral bounds for the superlevel set measures $|S_t|$,
since the inequality can be reformulated as
\[\int_0^x \big ( \one_A*\one_B)^\star(y)\,dy \le \int_0^x \big(\one_{A^\star}*\one_{B^\star}\big)(y)\,dy 
\text{ for all } x>0,\]
and the left-hand side equals $s|S_s|+\int_{t\ge s}|S_t|\,dt$ where $s=\big(\one_{A^\star}*\one_{B^\star}\big)(x)$.
The following example illustrates the nonexistence of useful upper bounds, in general, 
for the unintegrated quantities $|S_t|$.
Let $\lambda$ be a large positive integer. Choose sets $\scripta,\scriptb\subset\integers$ 
of cardinality $\lambda$, which satisfy $|\scripta+\scriptb|=|\scripta|\cdot|\scriptb|=\lambda^2$.
Define $A=\scripta+[-\tfrac12\lambda^{-1},\tfrac12\lambda^{-1}]\subset\reals$ and
$B=\scriptb+[-\tfrac12\lambda^{-1},\tfrac12\lambda^{-1}]\subset\reals$. 
Then $|A|=|B|=1$. 
Then $|S_t|=|\set{x:(\one_A*\one_B)(x)>t}|$ is equal to $0$ for $t\ge\lambda^{-1}$,
and equals $2\lambda(1-\lambda t)$ for $0<t<\lambda^{-1}$.
For two intervals $\tilde A,\tilde B$ of lengths equal to one, the corresponding
distribution function satisfies $|\tilde S_t|=2(1-t)$ for $t\in(0,1)$.
For all $t<(1+\lambda)^{-1}$, $|S_t|>|\tilde S_t|$; moreover, 
$|S_t|/|\tilde S_t|\asymp \lambda$ as $t\to 0$.

\section{A preliminary inverse Riesz-Sobolev inequality}

The special case in which
one of the three sets appearing in the expression $\langle \one_A*\one_B,\one_C\rangle$ is an interval
is simpler than the general case, but will be an essential step in our analysis.
We treat it here.

\begin{proposition} \label{prop:trivialintervalcase}
Let $K$ be a compact subset of $(0,\infty)$, and let $\eta>0$.
For each $\eps>0$ there exists $\delta>0$, depending also on $\eta,K$,
with the following property.
Let $A,B\subset\reals$ be measurable subsets with finite, positive measures,
and let $I\subset\reals$ be a bounded interval, such that $|A|,|B|,|I|$ all belong to $K$.
Assume further that for any permutation $(a,b,c)$ of $(|A|,|B|,|I|)$, $c\le (1-\eta)(a+b)$.
Suppose finally that
\begin{equation} \langle \one_A*\one_B,\one_I\rangle \ge (1-\delta)
\langle \one_{A^\star}*\one_{B^\star},\one_{I^\star}\rangle. \end{equation}
Then there exists an interval $J\subset\reals$ such that
\begin{equation} |A\bigtriangleup J|<\eps.  \end{equation}
\end{proposition}
This result is not formulated in a scale-invariant way, but via the action of the affine 
group it directly implies a scale-invariant generalization.
Theorem~\ref{mainthm} will later be deduced directly from Proposition~\ref{prop:trivialintervalcase}
and Corollary~\ref{cor:nearlyintervals}. 
Observe that in contrast to the setup of Theorem~\ref{mainthm},
$\one_A*\one_B,\one_I\rangle$ is assumed to be large for only one interval $I$, not for two. 

Fix $\eta,K$.  Proposition~\ref{prop:trivialintervalcase} is equivalent to the assertion
that if $(A_j,B_j,I_j)$ is a sequence of ordered triples satisfying
all of these hypotheses, with a sequence of parameters $\delta_j\to 0$, then there exist intervals $J_j$
such that $|A_j\bigtriangleup J_j|<\eps_j$ where $\eps_j\to 0$ as $j\to\infty$.
We prove this by contradiction.  If there were to exist a sequence for which the conclusion failed, 
then there would necessarily exist a subsequence for which
\begin{equation} (|A_j|,|B_j|,|I_j|)\to (\alpha,\beta,\gamma) \end{equation}
for some $(\alpha,\beta,\gamma)\in K^3$, so we may restrict attention to such a subsequence.\footnote{
The meaning of the symbols $\alpha,\beta$ here is unrelated to their role than in the statement
of Theorem~\ref{mainthm}.} 

Without loss of generality, we may assume that each interval $I_j$ is centered at $0$,
by translating $A_j,B_j,I_j$ by appropriate quantities.
Set $I = [-\tfrac12\gamma,\tfrac12\gamma]$.
Now \[\big| \langle \one_{A_j}*\one_{B_j},\one_I-\one_{I'}\rangle\big|\le C_K |I\bigtriangleup {I'}|\]
for any intervals $I,{I'}$. 
If $I,{I'}$ are centered at $0$, then $|I\bigtriangleup {I'}| = \big|\,|I|-|I'|\,\big|$.
Therefore $|I_j\cap I|\to 0$, and consequently
\begin{equation} \langle \one_{A_j}*\one_{B_j},\one_I\rangle
-\langle \one_{A_j}*\one_{B_j},\one_{I_j}\rangle\to 0 \end{equation} as $j\to\infty$.
Therefore we may replace $I_j$ by $I$ throughout the remainder of the discussion.

\begin{lemma} \label{lemma:nogaps}
There exist a function $\Lambda$ such that $\Lambda(r)\to 0$ as $r\to\infty$,
a sequence $R_j\to\infty$, and a sequence of real numbers $\tau_j$ such that 
\begin{equation} \label{eq:Ajdecaybound}
|A_j\setminus [\tau_j-\rho,\tau_j+\rho] | \le\Lambda(\rho) \text{ for all } \rho\in[0, R_j].  \end{equation}
\end{lemma}

\begin{proof}
If not, then after replacing the sequence of pairs $(A_j,B_j)$ by an appropriate subsequence,
there exists a sequence of bounded intervals $L_j=[\lambda_j^-,\lambda_j^+]$ such that
\begin{align*} &|A_j\cap L_j|\to 0, 
\\ &|L_j|\to\infty, \\ &\lim_{j\to\infty} |A_j^-| =\alpha^->0, 
\\ &\lim_{j\to\infty} |A_j^+|=\alpha^+>0, \end{align*}
where $A_j^-=A_j\cap (-\infty,\lambda_j^-]$ and
$A_j^+=A_j\cap [\lambda_j^+,\infty)$.
Since $\alpha^-+\alpha^+=\alpha$, both $\alpha_-,\alpha_+$ are strictly less than $\alpha$. 

Denote by $\tilde S$ the reflection of a subset $S\subset\reals$ about $0$.  
Decompose $\tilde B_j$ as the disjoint union
\[\tilde B_j=\tilde B_j^+\cup \tilde B_j^-\cup(\tilde B_j\cap [\lambda_j^-+|I|,\lambda_j^+-|I|]) \]
where 
\[ \tilde B_j^-=\tilde B_j\cap (-\infty,\lambda_j^-+|I|)\,\ \  \tilde B_j^+=\tilde B_j\cap (\lambda_j^+-|I|,\infty).\]
Write \[\langle \one_{A_j}*\one_{B_j},\one_I\rangle =\langle \one_{A_j}*\one_I,\one_{\tilde B_j}\rangle \]
where $\tilde B_j$ is the reflection of $B_j$ about $0$.
Then since $|A_j|,|B_j|$ belong to the fixed compact set $K$,
\begin{align*}\langle \one_{A_j}*\one_{I},\one_{\tilde B_j}\rangle
&= \langle (\one_{A_j^+}+\one_{A_j^-})*\one_{I},\one_{\tilde B_j}\rangle +O(|A_j\cap L_j|)
\\
&= \langle \one_{A_j^+}*\one_{I},\one_{\tilde B_j^+}\rangle 
+ \langle \one_{A_j^-}*\one_{I},\one_{\tilde B_j^-}\rangle 
+O(|A_j\cap L_j|)
\\ &\le\Theta(|A_j^-|,|B_j^-|,|I|) +\Theta(|A_j^+|,|B_j^+|,|I|) +O(|A_j\cap L_j|)
\\ &\to
\Theta(\alpha^-,\beta^-,\gamma) + \Theta(\alpha^+,\beta^+,\gamma)
\end{align*}
as $j\to\infty$.  The last inequality is justified by the Riesz-Sobolev inequality.  On the other hand, 
\begin{equation} \Theta(|A_j|,|B_j|,|I|)\ge
 \langle \one_{A_j}*\one_{B_j},\one_{I}\rangle\ge (1-\delta_j)\Theta(|A_j|,|B_j|,|I|)
\to\Theta(\alpha,\beta,\gamma).  \end{equation}
The left-hand side converges to $\Theta(\alpha,\beta,\gamma)$.  Therefore
\begin{equation} \label{impossibleidentity}
\Theta(\alpha,\beta,\gamma) = \Theta(\alpha^-,\beta^-\gamma) + \Theta(\alpha^+,\beta^+,\gamma),  \end{equation}
with $\alpha^-+\alpha^+=\alpha$, $\beta^-+\beta^+\le\beta$, and $\alpha^\pm\ne 0$.

This is impossible. Indeed, the right-hand side of \eqref{impossibleidentity} has the following interpretation.
Consider intervals $\scripti^\pm,\scriptj^\pm$ of lengths $\alpha^\pm,\beta^\pm$
respectively, such that $\distance(\scripti^-,\scripti^+)$ is sufficiently
large, $\scriptj^+$ has the same center as $\scripti^+$, and $\scriptj^-$ has the same center as $\scripti^-$.  Then
\begin{align*} \langle \one_{\scripti^+\cup\scripti^-}*\one_{\scriptj^+\cup\scriptj^-},\one_{I}\rangle
&= \langle \one_{\scripti^+}*\one_{\scriptj^+},\one_{I}\rangle
+\langle \one_{\scripti^-}*\one_{\scriptj^-},\one_{I}\rangle  
\\
&=\Theta(\alpha^+,\beta^+,\gamma)+\Theta(\alpha^-,\beta^-,\gamma).  \end{align*}
But 
\begin{equation} \langle \one_{\scripti^+\cup\scripti^-}*\one_{\scriptj^+\cup\scriptj^-},\one_{I}\rangle
<\Theta(|\scripti^+\cup\scripti^-|,|\scriptj^+\cup\scriptj^-|,|I|)= \Theta(\alpha,\beta,\gamma)\end{equation}
by Burchard's inverse theorem, since $\scripti^+\cup\scripti^-$ is not an interval. 
This contradicts \eqref{impossibleidentity}.  \end{proof}

So far, we have shown that the sets $A_j$ satisfy the decay bounds \eqref{eq:Ajdecaybound}.
The same reasoning applies to the sets $B_j$. By replacing $A_j$ by $A_j-\tau_j$,
we may assume henceforth that $\tau_j=0$.
One cannot simultaneously translate $A_j,B_j$ by independent amounts without disturbing the
hypothesis that $I$ is centered at $0$. But from that restriction on $I$, it now follows easily 
from the decay bounds for $B_j$ that the sequence $\tau'_j$ remains uniformly bounded,
hence that $B_j$ satisfies the same bounds with $\tau'_j\equiv 0$; otherwise necessarily
$\langle \one_{A_j}*\one_{B_j},\one_{I}\rangle\to 0$ as $j\to\infty$ for some subsequence of the indices $j$.

Next pass to a further subsequence, for which weak limits exist in $L^2$:
\begin{equation*} \one_{A_j}\rightharpoonup f \text{ and }  \one_{B_j}\rightharpoonup g\end{equation*}
for certain functions $f,g\in L^2(\reals)$.
By this we mean that for any test function $\varphi\in L^2(\reals)$, 
$\langle \one_{A_j},\varphi\rangle\to \langle f,\varphi\rangle$ as $j\to\infty$. 
Because $|A_j|,|B_j|$ belong to the compact set $K$, some subsequence must converge in this sense.
The uniform decay estimate \eqref{eq:Ajdecaybound}, in conjunction with the normalization $\tau_j\equiv 0$,
preclude the escape to spatial infinity of any mass, so $\norm{f}_1=\lim_{j\to\infty}|A_j|=\alpha$. 
Likewise, $\norm{g}_1=\lim_{j\to\infty}|B_j|=\beta$.  Moreover, $\norm{f}_\infty\le 1$ and $\norm{g}_\infty\le 1$.

\begin{lemma}
\begin{equation} \langle \one_{A_j}*\one_{B_j},\one_I\rangle\to \langle f*g,\one_I\rangle 
\ \ \text{ as $j\to\infty$.} \end{equation}
\end{lemma}

\begin{proof}
Let $\psi_i^\pm$ be continuous functions with ranges in $[0,1]$ such that
$\psi_i^-<\one_I<\psi_i^+$, $\psi_i^\pm$ is supported within distance $i^{-1}$ of $I$,
and $\psi_i^\pm\to\one_I$ from above and from below, respectively.
As $j\to\infty$,
$\langle \one_{A_j}*\one_{B_j},\psi_i^\pm\rangle
=\langle \one_{A_j}*\tilde\psi^\pm\rangle,\one_{\tilde B_j}\rangle \to \langle f*g,\psi_i^\pm\rangle$
for every $i$, since weak convergence of the sequence $\one_{A_j}$ in $L^2$ implies
strong $L^2$ convergence of the sequence $\one_{A_j}*\psi^\pm_i$  for fixed $i$.

Finally, let $i\to\infty$ and use the comparison
\[ \langle \one_{A_j}*\one_{B_j},\psi_i^-\rangle \le \langle \one_{A_j}*\one_{B_j},\one_I\rangle
\le\langle \one_{A_j}*\one_{B_j},\psi_i^+\rangle  \]
along with the corresponding upper and lower bounds for $\langle f*g,\one_I\rangle$.
\end{proof}

Therefore $\langle f*g,\one_I\rangle = \Theta(\alpha,\beta,\gamma)$.
Recall that $\norm{f}_\infty\le 1$, $\norm{g}_\infty\le 1$, $\norm{f}_1=\alpha$ and $\norm{g}_1=\beta$. 
The next lemma guarantees that under these circumstances, 
$f,g$ are indicator functions of sets of measures $\alpha,\beta$ respectively.

\begin{lemma}
Let $\scripta,\scriptb,I\subset\reals$ 
be intervals centered at $0$ of finite, positive lengths $|\scripta|,|\scriptb|,|I|$ which satisfy 
\begin{equation}\label{Btype} \max(|\scripta|,|I|)-\min(|\scripta|,|I|)< |\scriptb| < |\scripta|+|I|.\end{equation}
Let $f,g\in L^1(\reals)$ be nonnegative functions
satisfying $\norm{f}_\infty,\norm{g}_\infty\le 1$,
$\norm{f}_1=|\scripta|$ and $\norm{g}_1=|\scriptb|$.
Then  $\langle f*g,\one_I\rangle \le \langle\one_\scripta*\one_\scriptb,\one_I\rangle$.
Moreover, equality can  hold only if $f,g$ are indicator functions of sets of measures $|\scripta|,|\scriptb|$ respectively.
\end{lemma}

\begin{proof}
By the Riesz-Sobolev inequality,
$\langle f*g,\one_I\rangle \le \langle f^\star*g^\star,\one_I\rangle$.
Therefore it suffices to prove the result under the additional assumption that
$f=f^\star$ and $g=g^\star$,  which we assume henceforth.

Both $\one_\scripta,f$ are symmetric nonincreasing, and $f(x)\le \one_\scripta(x)$ for every $x\in\reals$,
so for any symmetric nonincreasing function $h$, $\int fh\le \int \one_\scripta h$.
Since $g*\one_I$ is a symmetric nonincreasing function,
\[\langle f*g,\one_I\rangle = \langle f,g*\one_I\rangle\le \int \one_\scripta\cdot (g*\one_I) =
\langle \one_\scripta*g,\one_I\rangle.\]
Repeating the argument with $f,g$ replaced by $g,\one_\scripta$ respectively
gives $\langle f*g,\one_I\rangle\le \langle \one_\scripta*\one_\scriptb,\one_I\rangle$.
Therefore
\[ \langle f*g,\one_I\rangle\le \langle \one_\scripta*g,\one_I\rangle 
\le\langle \one_\scripta*\one_\scriptb,\one_I\rangle.  \]

If $\langle f*g,\one_I\rangle = \langle \one_\scripta*\one_\scriptb,\one_I\rangle$, then the preceding inequality forces
$\langle \one_\scripta*g,\one_I\rangle = \langle \one_\scripta*\one_\scriptb,\one_I\rangle$.
Write $\langle \one_\scripta*g,\one_I\rangle$ as $\langle g,h\rangle$ where
$h=\one_\scripta*\one_I$ is symmetric nonincreasing,
and is strictly decreasing on the set of all $x$ which satisfy 
$\max(|\scripta|,|I|)-\min(|\scripta|,|I|)< 2|x|< |\scripta|+|I|$.
Under the assumption \eqref{Btype},
it is apparent that among all symmetric nonincreasing functions $g$ which satisfy
$\norm{g}_1=|\scriptb|$ and $\norm{g}_\infty\le 1$, 
$\int  gh$ is maximimized when $g=\one_\scriptb$, and in no other cases.
Therefore $g=\one_\scriptb$. By symmetry, $f=\one_\scripta$.
\end{proof}

We have shown so far that there are sets $A,B$ such that $\one_{A_j}\rightharpoonup\one_A$ where $|A_j|\to |A|$,
and likewise  $\one_{B_j}\rightharpoonup\one_B$ and $|B_j|\to |B|$.

\begin{lemma}
Let $E_j,E\subset\reals^d$ be Lebesgue measurable sets. Suppose that as $j\to\infty$, $|E_j|\to|E|<\infty$
and $\one_{E_j}\rightharpoonup \one_E$.
Then $|E_j\bigtriangleup E|\to 0$ as $j\to\infty$.
\end{lemma}

\begin{proof}
Let $\eps>0$. 
Let $K,\scripto$ respectively be a compact set and an open set such that 
$K\subset E\subset\scripto$ and $|\scripto\setminus K|<\eps$.
Let $\varphi:\reals^d\to[0,1]$ be a continuous function which satisfies $\varphi\equiv 1$ on $K$
and $\varphi\equiv 0$ outside of $\scripto$. Then
\[ |E_j\cap\scripto|\ge 
\int \varphi\one_{E_j}\to \int\varphi\one_E\ge \int \varphi\one_K=|K|>|E|-\eps.\] 
Therefore \[\liminf_{j\to\infty} |E_j\cap\scripto| \ge |E|-\eps,\] 
and consequently \[\liminf_{j\to\infty} |E_j\cap E| \ge |E|-\eps-|\scripto\setminus E|\ge |E|-2\eps.\] 
Since $|E_j|\to |E|$, this implies that $\limsup_{j\to\infty} |E_j\bigtriangleup E|<2\eps$.
\end{proof}

Thus $\one_{A_j}\to\one_A$ and $\one_{B_j}\to\one_B$ in $L^1$ norm.
Since we already know that $\langle \one_{A_j}*\one_{B_j},\one_I\rangle$
converges to $\langle f*g,\one_I\rangle=\langle \one_A*\one_B,\one_I\rangle$,
and since on the other hand
$\langle \one_{A_j}*\one_{B_j},\one_I\rangle\to\Theta(|A|,|B|,|I|)$ by hypothesis,
we conclude that
$\langle \one_A*\one_B,\one_I\rangle = \Theta(|A|,|B|,|I|)$.
These three measures $|A|,|B|,|I|$ satisfy the hypothesis \eqref{HBu} of 
Burchard's inverse theorem. Therefore $A,B$ are intervals, modulo null sets.
Since $|A_j \bigtriangleup A| = \norm{\one_{A_j}-\one_A}_1$ and the latter has been shown
to converge to zero, the proof of Proposition~\ref{prop:trivialintervalcase} is complete.
\qed

\medskip
To extend Proposition~\ref{prop:trivialintervalcase}
to higher dimensions, with the interval $I$ replaced by a compact convex set $K$ of positive Lebesgue measure, 
requires only a small modification. 
Let  $\scriptb(z,R)$ denote the ball in the norm associated to $K$, with center $z$ and radius $R$.
In place of Lemma~\ref{lemma:nogaps}, it suffices to show that there cannot exist a radius $R\in(0,\infty)$ and
center $z\in\reals^d$ such that $|A\cap \scriptb(z,R)|$ and $|A\cap (\reals^d\setminus \scriptb(z,2R))|$ are bounded below
while $|A\cap (\scriptb(z,2R)\setminus \scriptb(z,R))|$ is nearly equal to zero.
This follows from the proof of Lemma~\ref{lemma:nogaps}.
Precompactness is obtained by bounding $K$, inside and outside, by comparable ellipsoids,
then exploiting affine symmetries to reduce to the case where the ellipsoids are balls.

\section{Conclusion of proof}
\begin{proof}[Proof of Theorem~\ref{mainthm}]
Lemmas~\ref{lemma:Salphalowerbound}  and \ref{lemma:Salphaupperbound}
together with Corollary~\ref{cor:nearlyintervals}
demonstrate that $E$ is well approximated by some interval $I$,
in the sense that $|E\bigtriangleup I|\le C\delta^{1/2}\mab$.  Then 
\begin{align*} \langle \one_A*\one_B,\one_I\rangle 
&\ge \langle \one_A*\one_B,\one_E\rangle -\mab|E\bigtriangleup I|
\\
&\ge \Theta(|A|,|B|,|E|)-\delta\mab^2 - \mab |E\bigtriangleup I|
\\ &\ge \Theta(|A|,|B|,|I|) - C\delta^{1/2}\mab^2.  \end{align*}
By Proposition~\ref{prop:trivialintervalcase},
there exists an interval $J$ such that $|J\bigtriangleup A|<\eps$,
where $\eps\to 0$ as $\delta\to 0$.
\end{proof}

\section{Proof of Proposition~\ref{prop:keystone}} \label{section:keystone}
Write $\#(S)$ to denote the cardinality of a finite set $S$, and $|S|$ for the
Lebesgue measure of a subset $S\subset\reals$.
The proof of Proposition~\ref{prop:keystone} uses the following theorem of Freiman \cite{freiman}. 
See Theorem 5.11 of \cite{taovu} for an exposition,
and \cite{levsmeliansky} for an extension to two sets.

The theorem of Freiman states the following: 
Let $\scripta$ be a finite subset of $\integers$.
If $\#(\scripta+\scripta)<3\#(\scripta)-3$, then
$\scripta$ is contained in a rank one arithmetic progression of cardinality
$\le \#(\scripta+\scripta)-\#(\scripta)+1$.

\begin{proof}[Proof of Proposition~\ref{prop:keystone}]
Let $A\subset\reals$ be a Lebesgue measurable
set with finite, positive measure.  Assume that $|A+A|<3|A|-\rho$ for some $\rho>0$.
$\rho$ will remain fixed throughout the discussion.

Let $\eps,\delta>0$ be small parameters.
In particular, we require that $\delta<\tfrac12$.
For $n\in\integers$ consider the interval $I_n=(\eps n-\tfrac{\eps}2,\eps n+\tfrac{\eps}2)$.
Let $\scripta\subset\integers$ be the set of all $n$ for which $|A\cap I_n|\ge (1-\delta)|I_n|$,
and let $\tilde A=\cup_{n\in\scripta} I_n$.
By the Lebesgue differentiation theorem, the symmetric difference $A\,\bigtriangleup\,\tilde A$
satisfies $|A\,\bigtriangleup\, \tilde A|\to 0$ as $\max(\eps,\delta)\to 0$.
Therefore $\eps\#(\scripta)-|A|\to 0$ as $\max(\eps,\delta)\to 0$, as well.
In particular,
\[(1-\eta)\eps^{-1}|A|\le \#(\scripta)\le (1+\eta)\eps^{-1}|A|\] 
where $\eta=\eta(\eps,\delta)$ tends to zero as $\max(\eps,\delta)\to 0$.

If $S,T\subset(-\tfrac12,\tfrac12)$
are measurable sets of measures $>\tfrac12$, then $S\cap T$ has positive Lebesgue
measure and therefore $0\in S-T$.
It follows from this fact that if $m,n\in\scripta$,  then $\eps n+\eps m\in A+A$.
Therefore $k\in \scripta+\scripta\Rightarrow \eps k\in A+A$.
Therefore 
\begin{equation} \#(\scripta+\scripta)\le \eps^{-1}|A+A|. \end{equation}

Now
\begin{align*} \#(\scripta+\scripta) +3 &\le \eps^{-1}|A+A| +3 
\\ &< \eps^{-1}(3-\rho) |A| +3
\\ &\le  (3-\rho)(1-\eta)^{-1}\#(\scripta) +3  
\\ &=3\#(\scripta) +(3-\varrho\#(\scripta)) \end{align*}
where $\varrho>0$ may be taken to be independent of $\eps,\delta,\eta$ provided only
that these quantitites are sufficiently small.
Since $\#(\scripta)\to\infty$ as $\max(\eps,\delta)\to 0$,  $(3-\varrho\#(\scripta))<0$
provided that $\eps\delta$ are chosen to be sufficiently small, and thus $\#(\scripta+\scripta)<3\#(\scripta)-3$. 

The theorem of Freiman cited above now implies that there exists an arithmetic progression 
$\scriptp=\scriptp(\eps,\delta)\subset\integers$ such that $\scripta\subset\scriptp$ and
\[\#(\scriptp) \le \#(\scripta+\scripta)-\#(\scripta)+1.  \]
The set $P=P(\eps,\delta)= \cup_{n\in\scriptp} I_n$ then satisfies
\begin{multline*}|P| =\eps\#(\scriptp) \le \eps\#(\scripta+\scripta)-\eps\#(\scripta)+\eps  
\\ \le |A+A|-(1-\eta)|A|+\eps  = |A+A|-|A| +\eta|A|+\eps.  \end{multline*} 

$\scriptp$ is an arithmetic progression of rank $1$ in $\integers$, of some step $d$ which without
loss of generality can be taken to be positive. We claim that $d=1$.  Suppose not. Since
$\scriptp\subset k+d\integers$ for some $k\in\integers$, and $\scripta\subset\scriptp$, for any $m,m',n,n'\in\scripta$ 
\begin{equation}\label{disjointness} |(m+n)-(m'+n')|\ge 2 \text{ unless $m+n=m'+n'$.}\end{equation} 

Represent $A$ as the set of all $\eps n+\eps s$, where $n\in\scripta$ and $s\in S_n$,
where $S_n\subset (-\tfrac12,\tfrac12)$.
We have already arranged that $|S_n|\ge (1-\delta)$ for every $n\in\scripta$.

For any measurable sets $S,T \subset(-\tfrac12,\tfrac12)$, the associated sumset $S+T$
is contained in $(-1,1)$ and satisfies $|S+T|\ge |S|+|T|$ by the Brunn-Minkowski inequality.
Thus for each element $n$ of $\scripta+\scripta$, the set $A+A$ 
intersected with the interval of length $2\eps$ centered at $\eps n$
has measure $\ge (2-2\delta)\eps$. As $n$ varies over $\scripta+\scripta$, 
these intersections are pairwise disjoint by \eqref{disjointness}.
Since $\#(\scripta+\scripta)\ge 2\#(\scripta)-1$
by the Cauchy-Davenport inequality \cite{taovu}, 
\begin{equation*} |A+A| \ge (2-2\delta)\eps\#(\scripta+\scripta) \ge (2-2\delta)\eps(2\#(\scripta)-1). \end{equation*}

Therefore
\begin{equation*} |A+A| \ge (2-2\delta)(1-\eta)2|A| -2\eps \ge  (4-\varrho)|A| -2\eps \end{equation*}
where $\varrho\to 0$ as $\max(\eps,\delta)\to 0$.
For a sufficiently small choice of the parameters $\eps,\delta$,
this contradicts the hypothesis $|A+A|<3|A|$.  Therefore $d=1$.

The union of $P$ with finitely many points $n\pm\tfrac12$ is an interval.
We have thus proved that for any $\gamma>0$, there exists an interval $I_\gamma\subset\reals$
such that $A\subset I_\gamma$ and $|I_\gamma|\le |A+A|-|A|+\gamma$.
Then $I=\cap_{n=1}^\infty I_{1/n}$ is an interval; it contains $A$; it satisfies $|I|\le |A+A|-|A|$.
\end{proof}

\end{document}